\documentclass[12pt,leqno,amsfonts,amscd]{amsart}
\usepackage{epsfig}
\usepackage{amsmath}
\usepackage{amsfonts}
\usepackage{amssymb}
\usepackage{color}
\usepackage{hyperref}

\newtheorem{theorem}{Theorem}[section]
\newtheorem{lemma}[theorem]{Lemma}
\newtheorem{prop}[theorem]{Proposition}
\newtheorem{cor}[theorem]{Corollary}

\theoremstyle{definition}
\newtheorem{definition}[theorem]{Definition}

\newtheorem{question}[theorem]{Question}

\theoremstyle{remark}

\newtheorem{remark}[theorem]{Remark}
\newtheorem{remarks}[theorem]{Remarks}
\numberwithin{equation}{section}

\newcommand{\NN}{{\mathbb N}}

\newcommand{\RR}{{\mathbb R}}
\newcommand{\CC}{{\mathbb C}}

\newcommand{\out}[1]{\ }

\DeclareMathOperator{\MPSH}{MPSH}

\DeclareMathOperator{\QB}{QB}

\let\PSH=\psh

\let\cal=\mathcal
\renewcommand{\phi}{\varphi}

\begin{document}

\title[A note on maximal plurifinely plurisubharmonic functions]{A note on maximal plurifinely plurisubharmonic functions}



\author[M. El Kadiri]{Mohamed El Kadiri}
\address{R\'esidence Hasnae B1, Appt 14, rue chorafa,
Sal\'e-Tabriquet, Morocco}
\email{elkadiri@fsr.ac.ma}




\begin{abstract} In this note we study the plurifinely locally maximal
plurifinely plurisubharmonic functions and improve some known results on these functions. We prove in particular
that any locally bounded plurifinely locally maximal plurifinely plurisubharmonic function
is maximal, and that any finite plurisubharmonic function on an open subset of $\CC^n$
is maximal if and only if it is plurifinely locally maximal.
\end{abstract}
\maketitle

\section{Introduction}

The plurifine topology $\cal F$ on a Euclidean open set $\Omega\subset \CC^n$
is the smallest
topology that makes all plurisubharmonic functions on $\Omega$ continuous. This
construction is completely analogous to the better known fine topology in
classical potential theory of H. Cartan. Good references for the latter are
\cite{{AG}, {D}}. The topology $\cal F$ was introduced in \cite{F6}, and studied e.g. by Bedford and
Taylor in \cite{BT}, and by El Marzguioui and Wiegerinck in \cite{{E-W1},{EW1}}, where they
proved in particular that this topology is locally connected. Notions related to
the topology $\cal F$ are provided with the prefix $\cal F$, e.g. an $\cal F$-domain is an $\cal F$-open
set that is connected in $\cal F$.

Just as one introduces finely subharmonic functions on fine domains in $\RR^n$,
cf. Fuglede’s book \cite{F1}, one can introduce plurifinely plurisubharmonic functions
on $\cal F$-domains in $\CC^n$. These functions are called $\cal F$-plurisubharmonic
(or $\cal F$-psh in abbreviated form).
In case $n = 1$, we merely recover the finely subharmonic functions on fine
domains in $\RR^2$.

The definition of $\cal F$-plurisubharmonic functions on an $\cal F$-open set of $\CC^n$ was
first given in \cite{EK} and \cite{EW1}, where some properties of these functions were studied.
The $\cal F$-continuity of the $\cal F$-plurisubharmonic functions was established in \cite{EW2}.

In \cite{EFW}, the most important properties of the $\cal F$-plurisubharmonic functions
were obtained. This paper included a convergence theorem, and the characterization
of $\cal F$-plurisubharmonic functions as $\cal F$-locally bounded finely subharmonic
functions with the property that they remain finely subharmonic under
composition with $\CC$-isomorphisms of $\CC^n$. Hence the most important properties
of plurisubharmonic functions on Euclidean opens of $\CC^n$ were extended to
$\cal F$-plurisubharmonic functions on $\cal F$-open sets of $\CC^n$.

In \cite{EW}, the authors obtained a local approximation of $\cal F$-plurisubharmonic
functions by plurisubharmonic functions, outside a pluripolar set. They also
defined the Monge-Amp\`ere measure for finite $\cal F$-plurisubharmonic functions on
an $\cal F$-domain $U$. This construction was based on the fact (established in \cite{EW2})
that such a function can be $\cal F$-locally represented as a difference between two
bounded plurisubharmonic functions defined on some Euclidean open set, and
the quasi-Lindel\"{o}f property of the plurifine topology. The local approximation
property allowed them to prove that this Monge-Amp\`ere measure is a positive
Borel measure on $U$ which is $\cal F$-locally finite and does not charge the pluripolar
sets. It is $\sigma$-finite by the quasi-Lindel\"{o}f property of the plurifine topology.

In the theory of plurisubharmonic functions on a Euclidean domain, the so-called
maximal functions are the analog of the harmonic functions in classical
potential theory. They play an important role in the resolution of the Dirichlet
problem for the Monge-Amp\`ere operator.
In \cite{EK-MS}, El Kadiri and Smit introduced
the notion $\cal F$-maximal $\cal F$-plurisubharmonic
functions,
extending the notion of maximal plurisubharmonic functions on a Euclidean domain to an $\cal F$-domain,
and they studied
some basic properties of these functions.
They also studied the connection between $\cal F$-maximality of functions and the Monge-Amp\`ere operator.
In particular, they
proved that a finite $\cal F$-plurisubharmonic function $u$ on an $\cal F$-domain $U$
satisfies $(dd^cu)^n = 0$ if and only if $u$ is $\cal F$-locally $\cal F$-maximal outside some
pluripolar set.

The problem to know whether an $\cal F$-locally $\cal F$-maximal $\cal F$-plurisubharmonic function
is $\cal F$-maximal was settled in \cite{EK-MS}. In \cite{Ho} and \cite{Ho2},
it is proved that any
continuous or bounded $\cal F$-locally $\cal F$-maximal $\cal F$-plurisubharmonic function
is $\cal F$-maximal.
In the present note we shall prove that, more generally,
any locally bounded $\cal F$-locally $\cal F$-maximal $\cal F$-plurisubharmonic function
is $\cal F$-maximal. As a consequence of this result we prove that any finite $\cal F$-plurisubharmonic
$u$ function on an $\cal F$-open subset of $\CC^n$ is $\cal F$-locally
maximal if  and only if $(dd^cu)^n=0$, and that any finite $\cal F$-locally maximal
function on a Euclidean open subset  of $\CC^n$ is
maximal, improving some results in \cite{EK-MS}.
The last result cannot be improved as it is shown by the Example
4.19 in \cite{EK-MS}.

\section{Preliminaries}

Some elements of (resp. plurifine) pluripotential theory that will be used throughout the
paper can be found in \cite{BT,W}. The plurifine topology $\cal F$ on a Euclidean open set  of $\CC^n$ is
the smallest topology that makes all plurisubharmonic functions on it continuous.

For this article let $n$ be an integer $\ge 1$. If $A \subset \CC^n$, we denote the closure of
$A$ in the Euclidean, respectively plurifine topologies by $\overline A$ and ${\overline A}^{\cal F}$.
The Euclidean  and the $\cal F$ boundaries of $A$ are denoted by $\partial A$ and $\partial_{\cal F}A$
respectively. For a function $f$ on $A$ with values in $\overline \RR$, we denote by $\limsup_{x\in A,x\to y} f(x)$,
 and $\cal F$-$\limsup_{x\in A,x\to y} f(x)$ the $\limsup$ with respect to
the Euclidean topology of $\CC^n$  and the plurifine
topology of $\CC^n$
respectively, and likewise for other limits.

\begin{definition} Let $U$ be an $\cal F$ open set in $\CC^n$. A function
$u : U\longrightarrow  [-\infty, +\infty)$ is said to be
an $\cal F$-plurisubharmonic function if it is $\cal F$-upper semicontinuous, and for every complex
line $l$ in $\CC^n$, the restriction of $u$ to any $\cal F$ component of the finely open subset $l\cap U$  of $l$ is
either finely subharmonic or $= -\infty$.
\end{definition}

The set of all $\cal F$-plurisubharmonic functions in  $U$ is denoted by $\cal F$-$\PSH(U)$.

\begin{definition} Let  $U$ be an $\cal F$-open set in $\CC^n$ and let $\QB(U)$ be the trace of $\QB(\CC^n)$ on
$U$, where $\QB(\CC^n)$ denotes the $\sigma$-algebra on $\CC^n$ generated by the Borel sets and the
pluripolar subsets of $\CC^n$. Assume that $u \in \cal F$-$\PSH (U)$ is finite. Using the quasi-Lindel\"{o}f 
property of the plurifine topology and \cite[Theorem 2.17]{EW}, there exist a pluripolar set $E\subset U$,
a sequence of $\cal F$-open subsets $\{O_j\}$ and bounded plurisubharmonic functions $f_j , g_j$
defined on Euclidean open neighborhoods of $\overline{O_j}$ such that $U = E\cup\bigcup_{j=1}^\infty O_j$ and $u = f_j - g_j$ on
$O_j$. We define $O_0 := \emptyset$ and

\begin{equation}
\int_A(dd^cu)^n := \sum_{j=1}^\infty\int_{A\cap (Oj\setminus\bigcup_{k=0}^{j-1}O_k)} (dd^c(f_j - g_j ))^n, \ A \in \QB(U).
\end{equation}

By \cite[Theorem 3.6]{EW}, the measure defined by Eq. (2.1) is independent of $E$, $\{O_j\}$,
$\{f_j \}$ and $\{g_j\}$. This measure is called the complex Monge-Amp\`ere measure for the
$\cal F$-plurisubharmonic function $u$.
\end{definition}

\begin{theorem}[{\cite[Theorem 4.6]{EW}}] Let  $U$ be an $\cal F$-open set in $\CC^n$ and let $u\in \cal F$-$\PSH (U)$ be finite. Then
$(dd^cu)^n$ is a nonnegative $\cal F$-locally finite  measure on $\QB(U)$ (cf. \cite[Remark 4.7]{EW}).
\end{theorem}

\section{Maximal  and $\cal F$-locally maximal
${\cal F}$-plurisubharmonic functions}

In analogy with maximal plurisubharmonic functions, which play a role in
pluripotential theory comparable to that of harmonic functions in classical
potential theory, we will introduce $\cal F$-maximal $\cal F$-plurisubharmonic functions.
These relate similarly to finely harmonic functions and constitute the plurifine
analog of maximal plurisubharmonic functions on Euclidean open sets.

\begin{definition}\label{def2.1.} Let $U \subset \CC^n$
be an $\cal F$-open set and let $u \in {\cal F}$-$\PSH(U)$. We
say that $u$ is $\cal F$-maximal if for every bounded $\cal F$-open set $G$ of $\CC^n$
such that
${\overline G}\subset  U$, and for every function $v \in {\cal F}$-$\PSH(G)$ that is bounded from above on
$G$ and extends $\cal F$-upper semicontinuously to ${\overline G}^{\cal F}$, the following holds:
$v \le u \text{ on } \partial_{\cal F}G \Rightarrow v \le u \text{ on } G.$
\end{definition}
We denote by ${\cal F}$-$\MPSH(U)$ the set of $\cal F$-maximal $\cal F$-plurisubharmonic functions
on $U$.

\begin{theorem}[{\cite[Theorem 4.8]{EK-MS}}] Let  $U$ be an $\cal F$-open set in $\CC^n$ and let $u$ be a finite $\cal F$-maximal
$\cal F$-plurisubharmonic function in $U$. Then $(dd^cu)^n = 0$ in $\QB(U)$.
\end{theorem}

\begin{definition}\label{def3.12}  An ${\cal F}$-psh function $u$ in
an $\cal F$-open set $U$ of $\CC^n$ is said to be ${\cal F}$-locally maximal if each point
of $U$ has an ${\cal F}$-open neighborhood  $V$ such that
the restriction of  $u$ to $V$ is ${\cal F}$-maximal.
\end{definition}

It is clear that an $\cal F$- maximal $\cal F$-psh in an $\cal F$ domain $U$
of $\CC^n$ is ${\cal F}$-locally maximal
in $U$. The question to know whether the converse of this result is true was raised in
\cite{EK-MS}. A partial answer to this question is given by the following:

\begin{prop}[{\cite[Proposition 3.2]{Ho}}]\label{prop3.4}
Let $U$ be an $\cal F$-open set in $\CC^n$ and  $u\in {\cal F}$-$\PSH(U)$. If $u$ is a bounded  $\cal F$-locally maximal function
on $U$, then $u$ is $\cal F$-maximal on $U$.
\end{prop}

The following result is more general than the main theorems of \cite{Ho} and \cite{Ho2}:
\begin{theorem}\label{thm2.4}
Let $U$ be an $\cal F$-open set in $\CC^n$ and $u\in {\cal F}$-$\PSH(U)$.
If $u$ is locally bounded (relatively to the Euclidean topology)
$\cal F$-locally maximal function
on $U$, then $u$ is $\cal F$-maximal on $U$.
\end{theorem}

\begin{proof}
Suppose that $u$ is a locally bounded  $\cal F$-locally maximal function
in $U$. By the Lindel\"{o}f property  there is an increasing sequence $(V_j)$ of Euclidean open sets such that
$U\subset \bigcup_jV_j$ and $u$ is bounded on each $\cal F$-open set $V_j\cap U$.
Since for any $j$ the function $u$ is $\cal F$-locally maximal
and bounded on $V_j\cap U$, it
follows  from   Proposition \ref{prop3.4} that the function $u$ is $\cal F$-maximal on each $V_j\cap U$.
Now let $G$ be a bounded $\cal F$-open set of $\CC^n$ such that $\overline G\subset U$, then there exists $j\in \NN$ such that
$\overline G\subset V_j\cap U$. Let $v$ be an $\cal F$-upper semicontinuous function on
$\overline G^{\cal F}$ that is bounded from above and such that $v\in {\cal F}$-$\PSH(G)$ and $v\le u$ on
$\partial_{\cal F}G$, then
we have $v\le u$ on $G$ because $u$ is $\cal F$-maximal on $U\cap V_j$.
It follows that $u$ is $\cal F$-maximal.
\end{proof}

As corollary we obtain the following result of Hong and Viet \cite[Main Theorem]{Ho}:

\begin{cor}
Let $U$ be an $\cal F$-open set in $\CC^n$ and $u\in {\cal F}$-$\PSH(U)$. If $u$ is a continuous $\cal F$-locally maximal function
on $U$, then $u$ is $\cal F$-maximal on $U$.
\end{cor}

The following corollaries follow easily from Theorem \ref{thm2.4} and
\cite[Proposition 2.5, Theorem 4.1 and Theorem 4.14]{EK-MS}:

\begin{cor}\label{cor2.6}
Let $u$ be a locally bounded $\cal F$-psh function on an $\cal F$-open subset $U$ of $\CC^n$.
Then $u$ is $\cal F$-maximal if and only if $(dd^cu)^n=0$.
\end{cor}

\begin{cor}
Let $u$ be a continuous $\cal F$-psh function on an $\cal F$-open subset $U$ of $\CC^n$.
Then $u$ is $\cal F$-maximal if and only if $(dd^cu)^n=0$.
\end{cor}

\section{The case $n=1$}

For more details on finely superharmonic functions on a finely open subset of $\CC^n\cong \RR^{2n}$
we refer the reader to the monograph of Fuglede \cite{F1} on the subject.
We recall here the definition of the sweeping in a fine domain $U$ of $\CC$: let
$U$ a fine domain of $\CC$, $f:U\longrightarrow [0,+\infty]$ and $A\subset U$. Then we define
$$R_f^A:=\inf\{s\in \cal S^+(U)\cup\{+\infty\}: s\ge f \text{ on } A\}$$
and $\widehat R_f^A$ the finely lower semicontinuous regularized function of $R_f^A$ on $U$, that is the function
defined on $U$ by $\widehat R_f^A(z)$=f-$\liminf_{y\to z} R_f^A(z)$ for every $z\in U$, where
$\cal S^+(U)$ denotes the cone of nonnegative finely superharmonic functions on $U$.

A function $s\in \cal S^+(U)$ is called invariant if there exists a sequence $(V_j)$
of finely open subsets of $U$ such that $\bigcup_jV_j=U$, $\overline{V_j}\subset U$  and
$ \widehat R_u^{U\setminus V_j}=u$ for every $j$.

In \cite[Proposition 2.10]{EK-MS} we have proved that for any invariant function $u$ on
a finely open set of $\CC$, the function $-u$ is $\cal F$-locally maximal. Conversely,
we have the following:

\begin{theorem}\label{thm4.1}
Let $U$ be a regular finely open subset of $\CC$ and $u$ a nonnegative finely superharmonic function on $U$.
If $-u$ is maximal then $u$ is invariant.
\end{theorem}

\begin{proof}
We may suppose that $U$ is a fine domain of $\CC$.
Suppose that $-u$  maximal. By the Riesz decomposition
of $u$, there exist a fine potential $p$ on $U$ and an invariant function
$h$ on $U$ such that $u=p+h$.
By \cite[Theorem 4.4]{F4} and \cite[Corollaire, p. 132]{F5}, there exists a increasing sequence $(V_j)$
of finely open subsets of $U$ such that $\bigcup_jV_j=U$, $\overline{V_j}\subset V_{j+1}$
for each $j$, ${\inf}_j\widehat R_p^{U\smallsetminus V_j}=0$ q.e. (that is, outside of a polar subset of $U$) and
$R_h^{U\smallsetminus V_j}=h$ for any $j$.  Let $j$ be an integer $>0$ and $s$ a
finely superharmonic and nonnegative function on $U$ such that $s\ge u$ on $U\smallsetminus V_j$.
Since $-u$ is maximal we have $u\le s$ on $V_j$. We then deduce  that $u\le R_u^{U\smallsetminus V_j}$
and hence $u=R_u^{U\smallsetminus V_j}=\widehat R_u^{U\smallsetminus V_j}$. It follows that
$p=\widehat R_p^{U\smallsetminus V_j}$ for every $j$ and, by passing to the infimum over $j$, we obtain $p=0$ so that
$u=h$ is invariant.
\end{proof}

\begin{cor}
Let $u$ be a finely subharmonic function on a finely open subset $U$ of
$\CC$  bounded from above by some constant $c$. Then $u$ is maximal if and only $c-u$
is invariant.
\end{cor}

\begin{proof}
The corollary is an immediate consequence of the above theorem applied to
the function $c-u$.
\end{proof}

\begin{remarks} 1. It follows from \cite[Th\'eor\`eme 2.3]{F3} that any finely  harmonic function on $U$ is $\cal F$-maximal.
Conversely, any finite nonnegative $\cal F$-maximal function on $U$ is finely harmonic because any finite invariant function is
finely harmonic.

2. The Example 4.19 in \cite{EK-MS} shows that the converse of Proposition 2.10 in \cite{EK-MS} is not true.
However, we do not know if, in general, the converse of Theorem \ref{thm4.1} is true.
\end{remarks}

\section{$\cal F$-local maximality and $\cal F$-maximality
for finite $\cal F$-plurisubharmonic functions}

In \cite[Example 4.19]{EK-MS},
it is proved that
an $\cal F$-locally $\cal F$-maximal $\cal F$-psh function
need not be $\cal F$-maximal. In the present section we study
the $\cal F$-local maximality and $\cal F$-maximality
for finite $\cal F$-psh functions.

Let us recall the following result proven in  \cite{EK-MS}:

\begin{theorem}[{\cite[Theorem 4.15]{EK-MS}}]\label{thm3.1}
Let $U$ be an $\cal F$-open subset of $\CC^n$ and let $u\in \cal F$-$\PSH(U)$ be finite.
Then $(dd^cu)^n = 0$ if and only if $u$ is $\cal F$-locally $\cal F$-maximal on the complement
of an $\cal F$-closed pluripolar subset of $U.$
\end{theorem}

The following theorem improves Theorem \ref{thm3.1}:
\begin{theorem}\label{thm3.2}
Let $u$ be a finite $\cal F$-psh function on an $\cal F$-open subset $U$ of $\CC^n$.
Then $u$ is $\cal F$-locally $\cal F$-maximal if and only if $(dd^cu)^n=0$.
\end{theorem}

\begin{proof}
The part "only if" follows from Theorem \ref{thm3.1}. Let us prove the part "if".
Suppose that $(dd^cu)^n=0$. For any integer $j>0$ put $V_j=\{-j<u<j\}$. Then $V_j$ is
an $\cal F$-open and $u$ is bounded on $V_j$
and therefore
$u$ is maximal on $V_j$ by Corollary \ref{cor2.6}. Since $U=\bigcup_j V_j$, it follows that $u$ is
$\cal F$-locally $\cal F$-maximal on $U$.
\end{proof}

In a Euclidean open subset of $\CC^n$, it is proved in \cite{EK-MS} that  any
finite plurisubharmonic function  which is locally maximal is necessarily maximal (see  \cite[Theorem 4.22]{EK-MS}).
This leads us to ask the following question:

\begin{question}
Is any $\cal F$-locally $\cal F$-maximal finite $\cal F$-plurisubharmonic function
on an $\cal F$-open subset of $\CC^n$ $\cal F$-maximal?
\end{question}
The answer is  yes if $n=1$. Indeed a finite $\cal F$-locally
$\cal F$-maximal function on an $\cal F$-open $U$ subset of $\CC$ is just  finely locally finely harmonic
on $U$ and hence finely harmonic on $U$ by \cite[8.6, p. 70]{F1}, that is, $\cal F$-maximal on $U$.

The following theorem improves Theorem 4.22 in \cite{EK-MS}:

\begin{theorem}\label{thm5.4}
Let $u$ be a finite psh function on an open subset $U$ of $\CC^n$. Then $u$ is
maximal if and only if $u$ is $\cal F$-locally $\cal F$-maximal.
\end{theorem}

\begin{proof}
There is to prove only the part "if". Suppose that $u$ is
$\cal F$-locally $\cal F$-maximal, then $u$ is $\cal F$-maximal  on each $\cal F$-open
set $V_j=\{-j<u\}$, $j\in \NN^*$. Hence $u$ is maximal by
Theorem \ref{thm2.4} on each $V_j$ because it is locally bounded there ($u_j$ being
upper semicontinuous on $U$ and hence on $V_j$).
It follows that $(dd^cu)^n=0$ on $V_j$ for any $j$ so that
$(dd^cu)^n=0$ on $U$. For each $j$, let $u_j=\max\{u,-j\}$, then we have, as in the proof of \cite[Proposition 4.22]{EK-MS},
$\lim (dd^cu_j)^n=0$ weakly on $U$. It follows by \cite[Theorem 4.4]{BL1} that $u$ is maximal.
\end{proof}

\begin{remark}
We cannot drop the hypothesis that  $u$ is finite in the above theorem, see Example 4.19 in \cite{EK-MS}.
\end{remark}

\begin{cor}
Let $u$ be a finite
plurisubharmonic function on an open subset of $\CC^n$. Then $u$ is
maximal if and only if $(dd^cu)^n=0$.
\end{cor}

\begin{proof}
The part "only if" follows by Theorem 4.8 in \cite{EK-MS}. If $(dd^cu)^n=0$, then
$u$ is $\cal F$-locally maximal by Theorem \ref{thm3.2}, hence it is  maximal by the Theorem \ref{thm5.4}.
\end{proof}

In \cite{Ho2} it is proved that if $u$ is a bounded $\cal F$-psh
function on an $\cal F$-open subset $U$ of $\CC^n$, then for any integer $m\ge 1$ the function
$u\circ \pi$ is $\cal F$-maximal on $U\times \CC^m$ where $\pi$
is the canonical projection from $\CC^{n+m}$ into $\CC^n$. Here we generalize this result
and give a  simple and direct proof of it.

\begin{prop}\label{prop5.7}
Let $u$ be a finite $\cal F$-psh
function on an $\cal F$-open subset $U$ of $\CC^n$, then for any integer $m\ge 1$ the function
$u\circ \pi$ is $\cal F$-locally $\cal F$-maximal on $U\times \CC^m$, where $\pi$
is the canonical projection from $\CC^{n+m}$ into $\CC^n$.
\end{prop}

\begin{proof} For any $z\in U$, there exist by \cite [Theorem 3.1]{EW2}
an $\cal F$-open neighborhood $V_z$ of $z$ and a Euclidean
open set $O_z$ such that $z\in O_z$, $V_z\subset U$, and two bounded plurisubharmonic functions
$u_1$ and $u_2$ on $O_z$ such that $u=u_1-u_2$ on $O_z\cap V_z$. The functions $u_1\circ \pi$ and
$u_2\circ \pi$ are maximal plurisubharmonic functions on $O_z\times \CC^m$ by \cite[Proposition 3.5]{HH} and we have
$u\circ \pi=u_1\circ\pi - u_2\circ \pi$ on $(O_z\times \CC^m)\cap (V_z\times \CC^m)$ (this last set being an $\cal F$-open
subset of $\CC^{n+m}$).
Let $G$ be a bounded $\cal F$-open subset of $U\times \CC^m$ such that
$G\subset {\overline G}\subset (O_z\times \CC^m)\cap (V_z\times \CC^m)$ and 
let $v$ be a
$\cal F$-psh function on $G$ that can be extended upper semicontinuously to  ${\overline G}^{\cal F}$
denoted also by $v$ and such that $v\le u\circ\pi$ on $\partial_{\cal F}G$. Then we have
$v+u_2\circ \pi\le u_1\circ \pi$ on $\partial_{\cal F}G$ and hence $v+u_2\circ\pi \le u_1\circ\pi$ on $G$ since
$u_1\circ\pi$ is $\cal F$-maximal on $O_z\times \CC^m$ (and hence on $(O_z\times \CC^m)\cap (V_z\times \CC^m$)).
It follows that $v\le u\circ\pi$ on $G$. Hence the function $u\circ\pi$ is $\cal F$-locally
maximal on $U\times \CC^m$.
\end{proof}

\begin{cor}\label{cor5.9}
Let $u$ be a locally bounded  $\cal F$-psh
function on an $\cal F$-open subset $U$ of $\CC^n$, then for any integer $m\ge 1$ the function
$u\circ \pi$ is $\cal F$-locally $\cal F$-maximal on $U\times \CC^m$, where $\pi$
is the canonical projection from $\CC^{n+m}$ into $\CC^n$.
\end{cor}

\begin{proof}
The corollary is an immediate consequence of  Proposition \ref{prop5.7} and Theorem \ref{thm2.4}.
\end{proof}

\begin{lemma}\label{lemma5.9}
Let $(u_j)$ be a decreasing sequence of maximal psh functions on
an $\cal F$-open $U$ of $\CC^n$.
Then $u=\inf u_j$ is maximal on $U$.
\end{lemma}

\begin{proof}
Let $G$ be a bounded $\cal F$-open subset of $\CC^n$ such that $\overline G\subset U$ and let $v$ be
an $\cal F$-upper semicontinuous function on
$\overline G^{\cal F}$ that is bounded from above such that $v\in {\cal F}$-$\PSH(G)$
and $v\le u$ on  $\partial_{\cal F} G$.  Then
for any $j$ we have $v\le u_j$ on $G$ because $u_j$ is $\cal F$-maximal on $U$.
It follows then that $v\le u$ on $U$. Hence $u$ is $\cal F$-maximal on $U$.
\end{proof}

\begin{cor}\label{cor5.10}
Let $u$ is a locally upper bounded $\cal F$-psh
function on an $\cal F$-open subset $U$ of $\CC^n$, then for any integer $m\ge 1$ the function
$u\circ \pi$ is $\cal F$-locally $\cal F$-maximal on $U\times \CC^m$ where $\pi$
is the canonical projection from $\CC^{n+m}$ into $\CC^n$.
\end{cor}

\begin{proof}
Let $j\in \NN$, the function $u_j=\max(u,-j)$ is locally bounded on $U$, hence
$u_j\circ \pi$ is maximal in $U$  by  Corollary \ref{cor5.9} and Theorem \ref{thm2.4}.
It follows from Lemma \ref{lemma5.9} that $u\circ \pi=\inf_j(u_j\circ \pi)$ is $\cal F$-maximal in $U$.
\end{proof}

\begin{cor}
Let $v$ is an invariant function  on a finely open subset $U$ of $\CC$, then for any integer $m\ge 1$ the function
$u=-v\circ \pi$ is $\cal F$-locally $\cal F$-maximal on $U\times \CC^m$, where $\pi$
is the canonical projection from $\CC^{1+m}$ into $\CC$.
\end{cor}

\end{document}